\documentclass{article}
\usepackage{amssymb,amsfonts,amsmath,amsthm}
\newtheorem{thm}{Theorem}

\newtheorem{lem}[thm]{Lemma}

\theoremstyle{definition}

\theoremstyle{remark}

\begin{document}

\title{\centering{A Note On Free Boundary Minimal Annulus }}
\author{Shuangqi Liu \ \& \  Zuhuan Yu \thanks{School of Mathematical Science, Capital Normal University, Beijing 100048, China; yuzh@cnu.edu.cn}}





\maketitle

\begin{abstract}
  In this note we investigate free boundary minimal surfaces in the Euclidean 3-space, and by using holomorphic techniques developed by Fraser and Schoen we prove that the free boundary minimal annulus is the critical catenoid.
\end{abstract}



\section{Introduction}

The free boundary minimal surfaces come from the studying of partitioning of convex bodies and it has been studied for a long time. A classical result due to J. C. C. Nitsche [1] is that the minimal disc contained in the unit ball ${\mathbb{B}}^3$ in euclidean space and meeting the boundary $\partial {\mathbb{B}}^3$ orthogonally must be a flat equator, and it has an interesting generalization to the higher codimensions, which was obtained by Fraser and Schoen [2].

In recent years many free boundary minimal surfaces have been constructed out. Fraser and Scheon [3] constructed embedded surfaces of genus zero with any number of boundary components by finding the metrics on these surfaces that maximize the first Steklov eigenvalue with fixed boundary length. Pacard, Folha and Zolotareva [4] found examples of genus zero or one, and with any number of boundary components greater than a large constant. Martin Li and Kapouleas [5] constructed free boundary minimal surfaces with three boundary components and arbitrarily large genus.

There are lots of results of classification of free boundary minimal surfaces in unit ball ${\mathbb{B}}^3$ in euclidean 3-space. Ambrozio and Nunes [6] get a gap theorem for free boundary minimal surfaces which single out the flat disc and critical catenoid. In their paper [7] Fraser and Scheon prove that flat disc is the only free boundary minimal surface with Morse index equals to one and critical catenoid is the only free boundary minimal surface immersed by its first Steklov  eigenfunctions. Smith, Zhou [8], Tran [9] and Devyver [10] state that the Morse index of the critical catenoid equals to four.

Like the result of Nitsche for the flat equator disc, a natural question is whether a free boundary minimal annulus is the critical catenoid, which is a conjecture in [11, 12]. Under symmetric conditions MacGrath [13] shows the conjecture is true. Nadirashvili and Penskoi [14] prove the conjecture by using a connection between free boundary minimal surfaces and free boundary cones arising in a one-phase problem. In this note we prove it in a different way.

\begin{thm}{}
 A free boundary minimal immersed annulus in unit ball ${\mathbb{B}}^3$ in euclidean 3-space is congruent with the critical catenoid.
\end{thm}

\section{The critical Catenoid}

In this section we introduce the critical catenoid and compute the first and second fundament form. Let $\Omega$ be an annulus domain in complex plane $\mathbb{C}$, assume $\Omega = \{z|z=x+yi=re^{i\theta}, 0<r_1\leq r \leq r_2, 0\leq \theta < 2\pi\}.$ The critical Catenoid is the conformal minimal immersion $X: \Omega \rightarrow R^3$ given by
$$ X(z)=\{ a\cosh(\ln r)cos\theta, a\cosh(\ln r)\sin \theta, a\ln r\},$$
and the parameters $a, r_1, r_2$ are determined by the following relations
 $$ r_1r_2=1, \  \ r^2_2+1=(r^2_2-1)\ln r_2, \ \  a=2r_2((r^2_2+1)\ln r_2)^{-1}.$$
Through a direct computation, we have the first and second fundamental forms in polar coordinates as follows
$$ I=Edr^2+Gd\theta^2,\ \ E=\frac{a^2}{r^2}\cosh^2(\ln r),\ \ F=0, \ \ G=a^2\cosh^2 (\ln r),$$
$$II=Ldr^2+Nd\theta^2,\ \ L=-\frac{a}{r^2}, \ \ M=0, \ \ N=a.$$
Here we choose the normal vector fields insuring $N$ is positive, and from the results $F=M=0$, we know that coordinate lines ($r-$line and $\theta-$line) are lines of curvature, namely tangent directions of the coordinate curves are principle directions.
\section{The second fundamental form of the annulus}
 Let us consider the free boundary minimal annulus in the unit ball $\mathbb{B}^3$ in Euclidean 3-space. By taking the global isothermal coordinate, using the conformal and minimal condition, and appealing the technique due to Fraser and Scheon, we can show the second fundamental form is the same as that form of the critical catenoid.

 The free boundary minimal annulus is given by the conformal immersion $U: \Omega \rightarrow \mathbb{B}^3 \subset R^3, U(\partial \Omega)\subset \mathbb{S}^2=\partial \mathbb{B}^3$. Consider the conformal structure of the annulus, we can take $\Omega = \{z|z=x+yi=re^{i\theta}, 0<R^{-1}\leq r \leq R, R>1, 0\leq \theta < 2\pi\}.$ The induced metric is written as
 $$
 I=\lambda(dzd\bar{z})=\lambda (dr^2+r^2d\theta^2)=Edr^2+2Fdrd\theta+Gd\theta^2, $$  $$\\  E=\lambda,\  F=0, \  G=\lambda r^2.
 $$
 Now we apply the methods used by Fraser and Scheon in their paper [2]. By the conformal and minimal condition we have
 $$
 U_z\cdot U_{z} =0,\ \  U_{z\bar{z}}=0.
 $$
One can easily check that $(U_{zz}\cdot U_{zz})_{\bar{z}}=0$, therefore $U_{zz}\cdot U_{zz}$ is a holomorphic function. By $0=(U_z\cdot U_z)_z=2U_{zz}\cdot U_z$, $U_{zz}$ can be decomposed as
$$
U_{zz}=U_{zz}^{\bot}+xU_z,\ \  x=\frac{U_{zz}\cdot U_{\bar{z}}}{|U_z|^2},
$$
$U_{zz}^{\bot}$ is the projection part on normal space, and this implies $ U_{zz}^{\bot}\cdot U_{zz}^{\bot}=U_{zz}\cdot U_{zz},$ then the inner product of the normal part $ U_{zz}^{\bot}\cdot U_{zz}^{\bot}$ is also a holomorphic function.

Appealing to the polar coordinates $(r, \theta)$, we rewrite the $U_{zz}$ as
$$
U_{zz}=e^{-2i\theta}(U_{rr}-\frac{1}{r^2}U_{\theta \theta}-\frac{2i}{r}U_{\theta r}+\frac{2i}{r^2}U_{\theta}-\frac{1}{r}U_r),
$$
Thus
\begin{align*}
  z^4 U_{zz}^{\bot}\cdot U_{zz}^{\bot} &=(r^2U_{rr}^{\bot}-U_{\theta \theta}^{\bot}-2irU_{\theta r}^{\bot})^2,\\
                                       &=((r^2L-N-2irM)n)^2,\\
                                       &=(r^2L-N-2irM)^2,\\
                                       &=(r^2L-N)^2-4r^2M^2-4rM(r^2L-N)i,
\end{align*}
which is a holomorphic function on $\Omega$, where $n$ is the normal vector fields, while $L, M, N$ are the quantities of the second form.

In the following we prove the main theorem of this section. According the above coordinates the second fundamental form of the minimal annulus can be clearly formulated in bellow.
\begin{lem}
The second fundamental form of the free boundary minimal annulus can be written as
 $$
 II=Ldr^2+2Mdrd\theta+Nd\theta^2,\ \ \ L=-\frac{A}{r^2},\ \ M=0,\ \ N=A,
 $$
 here A is a positive constant number.
 \end{lem}
 \begin{proof}
 By the free boundary condition,$U_r$ is orthogonal to $\partial \mathbb{B}^3$ along the boundary of the annulus ($\partial \Omega $), so $U_r =k U$, and $k$ is defined on the boundary $\partial \Omega $, then
 $$
U_{r\theta} =k_{\theta}\ U+kU_{\theta}=\frac{k_{\theta}}{k}u_r+kU_{\theta},
 $$
  and so $U^{\bot}_{r\theta} =0$, namely $M=0$ on $\partial \Omega $, this implies the holomorphic function  $z^4 U_{zz}^{\bot}\cdot U_{zz}^{\bot}$ on $\Omega$ takes real values on the boundary $\partial \Omega,$ then it must be constant in $\bar {\Omega}$. We get
  \begin{align*}
  (r^2L-N)^2-4r^2M^2&=(r^2L-N)^2|_{\partial \Omega}=const.\geq0,\\
   -4rM(r^2L-N)&=0,
   \end{align*}
 and both of which hold in $\bar{\Omega}$. The constant number must be positive. If not, $const.=0$, by minimal condition $EN+GL=0$, we have
   \begin{align*}
  (r^2L-N)^2-4r^2M^2&=0,\\
   -4rM(r^2L-N)&=0,\\
   \lambda N+\lambda r^2L&=0,
   \end{align*}
   since $\lambda >0, \ 0<R^{-1}<r<R$, hence $L=M=N=0$ in $\Omega$, which means the free boundary minimal annulus is totaly geodesic and flat, which can not occur. So in the closed domain $\bar{\Omega}$, we have
   \begin{align*}
  (r^2L-N)^2-4r^2M^2&=const.>0,\\
   -4rM(r^2L-N)&=0,\\
   \lambda N+\lambda r^2L&=0,
   \end{align*}
   by the first equation, $(r^2L-N)^2 >0$, then $M=0$ from the second equation, so
   \begin{align*}
  r^2L-N&=const.\neq 0,\\
     N+r^2L&=0,
    \end{align*}
    we obtain that $N=A >0$, by choosing a suitable normal vector field, $L=-\frac{A}{r^2}$, and the lemma is proved.
\end{proof}

\section{The Gauss equation and Gauss map of the annulus}

From the last section we know that the Gauss curvature of the minimal annulus reads
$$
 k=-\frac{A^2}{\lambda^2 r^4}.
$$
On the other hand we have
$$
 k=\frac{1}{\lambda}\Delta (\ln \frac{1}{\sqrt{\lambda}}),
$$
where
$$
\Delta=\frac{\partial^2}{\partial x^2}+\frac{\partial^2}{\partial y^2}\ \\ =4\frac{\partial^2}{\partial z\partial {\bar z}}\ \\ =\frac{\partial^2}{\partial r^2}+\frac{1}{r}\frac{\partial}{\partial r}+\frac{1}{r^2}\frac{\partial^2}{\partial \theta^2}.
$$
Denote $\phi=\ln \frac{1}{\sqrt \lambda}$, so the Gauss equation is written as
\begin{equation}
\Delta \phi+\frac{A^2}{r^4}e^{2\phi}=0.
\end{equation}

Now we apply the Weierstrass formula of the minimal annulus. Let $(g,fdz)$ be the Weierstrass data on the annulus, in which $g$ is the Gauss map of the minimal annulus which is a meromorphic function and $f$ is holomorphic function. Both of them satisfy
\begin{enumerate}
  \item  a pole point of order l of $g$ is exactly a zero point of order 2l of $f$,
  \item  the real periods of the forms $((1-g^2)fdz, i(1+g^2)fdz, 2gfdz)$ are zero.
\end{enumerate}
The minimal immersion $U: \Omega \rightarrow R^3$ is
\begin{equation}
U(z)=\text{Re} \int \{(1-g^2)fdz, i(1+g^2)fdz, 2gfdz\}£¬
\end{equation}
and the metric $I=|f(z)|^2(1+|g(z)|^2)^2dzd{\bar z}$, then $\lambda=|f(z)|^2(1+|g(z)|^2)^2, $ hence $\phi=-\frac{1}{2}\ln |f|^2-\frac{1}{2}\ln (1+|g|^2)^2.$ By a direct computation, one shows
\begin{align*}
\Delta \phi =-4\frac{|g_z|^2}{(1+|g|^2)^2},\ \
e^{2\phi} =\frac{1}{\lambda}=\frac{1}{|f|^2(1+|g|^2)^2}.
\end{align*}
Using the Gauss equation (1) we get $2|fg_zz^2|=|A|,$ thus $$f(z)=\frac{A}{2}\frac{e^{i\theta_0}}{g_zz^2},$$
here $\theta_0$ is a constant real number, and by the condition 1, $g$ only has pole and zero points of order 1. In summary we have
\begin{thm}
The free boundary minimal annulus is determined by its Gauss map and the Weierstrass data can be written as
 \begin{equation}
 (g,\omega)=(g, \frac{A}{2}\frac{e^{i\theta_0}}{g_zz^2}dz).
 \end{equation}
 \end{thm}

 \section{The boundaries of the free boundary minimal annulus}

  In this section by using the Weierstrass representation of the minimal annulus and the boundary conditions we prove that the boundaries are planar circles.
  \begin{lem}
  The boundary curves of the free boundary minimal annulus are planar circles.
  \end{lem}
  \begin{proof}
  For any point $P$ on the boundary $\partial \Omega$, its conformal coordinate is denoted by $z_0=Re^{i\theta_1},$ (or $ z_0=R^{-1}e^{i\theta_1}$ ). We compute the torsion of the boundary curve on the point $P$.

  For simplicity, if needed, by taking a rotation of the ball $\mathbb{B}^3$, we can put $P(z_0)$ at a special position such that the tangent vector at point $P(z_0)$ of boundary curve is parallel to $Y$ axis, and the unit normal vector of minimal annulus at the point $P(z_0)$ is located in $XZ$ plane, properly,
 $$
  \frac{dU(\theta)}{d\theta}|_{\theta=\theta_1} \parallel (0,1,0),\ \ g(z_0)=R.
 $$
 With the assumption that $z_0$ is not a pole, then $g(z)$ is holomorphic around the point $z_0$. For convenience of computation we choose another complex chart around $P$: Let $z=e^w$, i.e. $w=u+iv=\ln z=\ln r+i\theta,\ \ u=\ln r ,\ v=\theta$, $w_0=\ln z_0=\ln R+i\theta_1$. Assume that the minimal annulus can be extended across the boundary defined on a region for a small $\delta>0$, $$\Omega^*=\{w=u+iv|\ln R-\delta<u<\ln R+\delta, \ \theta_1-\delta<\theta<\theta_1+\delta\}.$$
  The minimal annulus on this region $\Omega^*$ has the Weierstrass data
  $g(z)=g(e^w)$ (also denoted by $g(w)$), which is holomorphic and  $$\frac{dg(w)}{dw}=\frac{dg(z)}{dz}e^w,\ \ \
\omega=\frac{A}{2}\frac{e^{i\theta_0}}{g_zz^2}dz=\frac{A}{2}\frac{e^{i\theta_0}}{g_w}dw.$$
Since $g(w)$ is holomorphic on domain $\Omega^*$ around $w_0$ and $g(w_0)=g(z_0)=R$, then $g$ and $g_w$ can be expressed in series of
$$
  \begin{array}{ll}
    g=R+a_1(w-w_0)+a_2(w-w_0)^2+a_3(w-w_0)^3+\cdots , \\
    g_w=a_1+2a_2(w-w_0)+3a_3(w-w_0)^2+\cdots .
  \end{array}
$$
As $\omega$ is holomorphic on the domain $\Omega^*$ around $w_0$ hence $a_1\neq 0$. Consequently we get
\begin{align*}
\frac{1}{g_w} =&\frac{1}{a_1}-\frac{2a_2}{a_1^2}(w-w_0)+(\frac{4a_2^2}{a_1^3}-\frac{3a_3}{a_1^2})(w-w_0)^2+\\
&(-\frac{8a_2^3}{a_1^4}+\frac{12a_2a_3}{a_1^3}-\frac{4a_4}{a_1^2})(w-w_0)^3+\cdots,\\
\frac{g^2}{g_w} =&\frac{R^2}{a_1}+2R(1-\frac{a_2}{a_1^2}R)(w-w_0)+[(\frac{4a_2^2}{a_1^3}-\frac{3a_3}{a_1^2})R^2-\frac{2a_2}{a_1}R+a_1](w-w_0)^2+\\
&[(-\frac{8a_2^3}{a_1^4}+\frac{12a_2a_3}{a_1^3}-\frac{4a_4}{a_1^2})R^2+4(\frac{a_2^2}{a_1^2}-\frac{a_3}{a_1})R](w-w_0)^3+\cdots,\\
\frac{g}{g_w} =&\frac{R}{a_1}+(1-\frac{2a_2}{a_1^2}R)(w-w_0)+[(\frac{4a_2^2}{a_1^3}-\frac{3a_3}{a_1^2})R-\frac{a_2}{a_1}](w-w_0)^2+\\
&[(-\frac{8a_2^3}{a_1^4}+\frac{12a_2a_3}{a_1^3}-\frac{4a_4}{a_1^2})R+2(\frac{a_2^2}{a_1^2}-\frac{a_3}{a_1})](w-w_0)^3+\cdots.\\
\end{align*}
Now let $(X,Y,Z)$ denote the coordinates of points in $R^3$, and take the Weierstrass representation of the annulus on the domain $\Omega^*$, $U:\Omega^*\rightarrow R^3$
\begin{align*}
 X(w)&=\text {Re}\int^w_{w_0}(1-g(w)^2)\frac{A}{2}\frac{e^{i\theta_0}}{g_w}dw+X_0,\\
Y(w)&=\text {Re}\int^w_{w_0}i(1+g(w)^2)\frac{A}{2}\frac{e^{i\theta_0}}{g_w}dw+Y_0,\\
Z(w)&=\text {Re}\int^w_{w_0}2g(w)\frac{A}{2}\frac{e^{i\theta_0}}{g_w}dw+Z_0.
\end{align*}
So we get
\begin{align*}
 X(w)=&X_0+\frac{A}{2}\text {Re} \ e^{i\theta}\{\frac{1-R^2}{a_1}(w-w_0)-[(1-\frac{a_2}{a_1^2}R)R+\frac{a_2}{a_1^2}](w-w_0)^2\\
      &+\frac{1}{3}[(\frac{4a_2^2}{a_1^3}-\frac{3a_3}{a_1^2})(1-R^2)+\frac{2a_2}{a_1}R-a_1](w-w_0)^3+\cdots\},\\
 Y(w)= &Y_0+
\frac{A}{2}\text {Re} \ ie^{i\theta}\{\frac{1+R^2}{a_1}(w-w_0)+[(1-\frac{a_2}{a_1^2}R)R-\frac{a_2}{a_1^2}](w-w_0)^2\\
       &+\frac{1}{3}[(\frac{4a_2^2}{a_1^3}-\frac{3a_3}{a_1^2})(1+R^2)-\frac{2a_2}{a_1}R+a_1](w-w_0)^3+\cdots\},\\
 Z(w)= &Z_0+\frac{A}{2}\text {Re} \ 2e^{i\theta}\{\frac{R}{a_1}(w-w_0)+\frac{1}{2}(1-\frac{2a_2}{a_1^2}R)(w-w_0)^2\\
       &+\frac{1}{3}[(\frac{4a_2^2}{a_1^3}-\frac{3a_3}{a_1^2})R-\frac{a_2}{a_1}](w-w_0)^3+\cdots\}.\\
\end{align*}
On the boundary curve, $w=\ln R+i\theta, w_0=\ln R+i\theta_1$, therefore $w-w_0=i(\theta-\theta_1)$, appealing the above equations we have the vectors at the point $P(w_0)$ as
$$\left(
    \begin{array}{c}
      X_\theta \\
      Y_\theta \\
      Z_\theta \\
    \end{array}
  \right)(w_0)=\frac{A}{4}\left(
                            \begin{array}{c}
                              (1-R^2)(\frac{e^{i\theta_0}}{a_1}-\frac{e^{-i\theta_0}}{\bar {a}_1})i \\
                             -(1+R^2)(\frac{e^{i\theta_0}}{a_1}+\frac{e^{-i\theta_0}}{\bar {a}_1})\\
                              2R(\frac{e^{i\theta_0}}{a_1}-\frac{e^{-i\theta_0}}{\bar {a}_1})i\\
                            \end{array}
                          \right)\parallel \left(
                                             \begin{array}{c}
                                               0 \\
                                               1 \\
                                               0 \\
                                             \end{array}
                                           \right),$$
hence $\frac{e^{i\theta_0}}{a_1}$ is a real number, in other words $a_1=\pm|a_1|e^{i\theta_0},$ and (from boundary assumption at $P(w_0)$)
$$\left(
    \begin{array}{c}
      X_u \\
      Y_u \\
      Z_u \\
    \end{array}
  \right)(w_0)=\frac{A}{2}\left(
                            \begin{array}{c}
                              (1-R^2)\frac{e^{i\theta_0}}{a_1} \\
                               0  \\
                              2R\frac{e^{i\theta_0}}{a_1}\\
                            \end{array}
                          \right) \parallel \left(
                                             \begin{array}{c}
                                               X_0 \\
                                               Y_0 \\
                                               Z_0 \\
                                             \end{array}
                                           \right),$$
so the position vector of $P(w_0)$
$$\left(
                                             \begin{array}{c}
                                               X_0 \\
                                               Y_0 \\
                                               Z_0 \\
                                             \end{array}
                                           \right)=\frac{1}{1+R^2}\left(
                            \begin{array}{c}
                              1-R^2 \\
                               0  \\
                              2R \\
                            \end{array}
                          \right).$$

Firstly, for convenience denoting $\Psi=\frac{4a_2^2}{a_1^2}-\frac{3a_3}{a_1},$ the coordinate functions of the boundary curve near $P$ are
\begin{align*}
X= &X_0+\frac{A}{4}\frac{e^{i\theta_0}}{a_1}\{[(\frac{a_2}{a_1}+\frac{\bar a_2}{\bar a_1})(1-R^2)+R(a_1+\bar a_1)](\theta-\theta_1)^2\\
   &+\frac{i}{3}[(\bar \Psi-\Psi)(1-R^2)+2R(\bar {a}_2-a_2)+a_1^2-\bar {a}_1^2](\theta-\theta_1)^3+\cdots \},\\
Y= &\frac{A}{4}\frac{e^{i\theta_0}}{a_1}\{-2(1+R^2)(\theta-\theta_1)+i[(\frac{a_2}{a_1}-\frac{\bar a_2}{\bar a_1})(1+R^2)+R(\bar a_1- a_1)](\theta-\theta_1)^2\\
   &+\frac{1}{3}[(\bar \Psi+\Psi)(1+R^2)-2R(\bar {a}_2+a_2)+a_1^2+\bar {a}_1^2](\theta-\theta_1)^3+\cdots \},\\
Z= &Z_0+\frac{A}{2}\frac{e^{i\theta_0}}{a_1}\{[(a_1+\bar a_1)-2R(\frac{a_2}{a_1}+\frac{\bar a_2}{\bar a_1})](\theta-\theta_1)^2\\
   &+\frac{i}{3}[(\bar \Psi-\Psi)R+a_2-\bar {a}_2](\theta-\theta_1)^3+\cdots \}.\\
\end{align*}

 The normal vector field of the boundary near the point $P$ can be written as
\begin{align*}
X_u= &\frac{A}{4}\frac{e^{i\theta_0}}{a_1}\{2(1-R^2)+2i[(\frac{\bar a_2}{\bar a_1}-\frac{a_2}{a_1})(1-R^2)+R(\bar a_1-a_1)](\theta-\theta_1)\\
     &-[(\bar \Psi+\Psi)(1-R^2)+2R(\bar {a}_2+a_2)-(a_1^2+\bar {a}_1^2)](\theta-\theta_1)^2+\cdots \},\\
Y_u= &\frac{A}{4}\frac{e^{i\theta_0}}{a_1}\{2[(\frac{\bar a_2}{\bar a_1}+\frac{a_2}{a_1})(1+R^2)-R(\bar a_1+a_1)](\theta-\theta_1)\\
     &+i[(\bar \Psi-\Psi)(1+R^2)+2R(a_2-\bar {a}_2)-a_1^2+\bar {a}_1^2](\theta-\theta_1)^2+\cdots \},\\
Z_u= &\frac{A}{2}\frac{e^{i\theta_0}}{a_1}\{2R+i[2(\frac{\bar a_2}{\bar a_1}-\frac{a_2}{a_1})R-\bar a_1+a_1](\theta-\theta_1)\\
      &+[(\bar \Psi+\Psi)R-(\bar {a}_2+a_2)](\theta-\theta_1)^2+\cdots \}.\\
\end{align*}

Next we recall the boundary orthogonal assumption, which is equivalent to $(X,Y,Z)\parallel(X_u,Y_u,Z_u)$, or
\begin{equation}
\frac{X_u}{X} =\frac{Y_u}{Y} =\frac{Z_u}{Z}.
\end{equation}

From the first equality of (4), i.e. $YX_u =XY_u$, comparing the coordinates of $(\theta-\theta_1)$ and $(\theta-\theta_1)^2$ on two sides of the equation we get
\begin{equation}
A\frac{e^{i\theta_0}}{a_1}= -\frac{2}{(1+R^2)^2}\{(\frac{\bar a_2}{\bar a_1}+\frac{a_2}{a_1})(1+R^2)-R(\bar a_1+a_1)\},
\end{equation}
\begin{equation}
{\bar \Psi}-\Psi=\Theta_1-\Theta_2
 \end{equation}
$$\Theta_1=\frac{A}{2}\frac{e^{i\theta_0}}{a_1}\{3(\frac{a_2}{a_1}-\frac{\bar a_2}{\bar a_1})(1+R^2)-\frac{R(1+3R^2)}{1-R^2}(\bar a_1-a_1)\},$$
$$ \Theta_2=\frac{1}{1+R^2}\{2R(a_2-\bar {a}_2)-a_1^2+\bar {a}_1^2\}.$$

In the same way, comparing the coordinates of $(\theta-\theta_1)$ and $(\theta-\theta_1)^2$ on two sides of the second equality of (4), i.e. $YZ_u =ZY_u$, we obtain the equation (5) once again and
\begin{equation}
{\bar \Psi}-\Psi=\Pi_1-\Pi_2,
\end{equation}
$$\Pi_1=\frac{A}{2}\frac{e^{i\theta_0}}{a_1}\{3(\frac{a_2}{a_1}-\frac{\bar a_2}{\bar a_1})(R^2+1)+(\frac{1}{R}+2R)(\bar a_1-a_1)\},$$
 $$  \Pi_2=\frac{1}{1+R^2}\{2R(a_2-\bar {a}_2)-a_1^2+\bar {a}_1^2\}.$$

On the other hand the boundary curve is on the sphere, so $X^2+Y^2+Z^2=1$. Comparing the coordinates of $(\theta-\theta_1)^2$ and $(\theta-\theta_1)^3$ on two sides of the equation we also get
\begin{equation}
A\frac{e^{i\theta_0}}{a_1}= -\frac{2}{(1+R^2)^3}[(\frac{\bar a_2}{\bar a_1}+\frac{a_2}{a_1})(1+R^4)-R^3(\bar a_1+a_1)],
\end{equation}

\begin{equation}
{\bar \Psi}-\Psi=\Upsilon_1-\Upsilon_2,
\end{equation}
$$\Upsilon_1=\frac{3A}{2}\frac{e^{i\theta_0}}{a_1}\{(\frac{a_2}{a_1}-\frac{\bar a_2}{\bar a_1})\frac{(1+R^2)^3}{1+R^4}+\frac{R(1+R^2)^2}{1+R^4}(\bar a_1-a_1)\},$$
  $$\Upsilon_2=\frac{2R^3}{1+R^4}(a_2-\bar {a}_2)+\frac{1-R^2}{1+R^4}(a_1^2-\bar {a}_1^2).$$

To combine the equation (5) with (8) we have
\begin{equation}
\frac{a_2}{a_1}+\frac{\bar a_2}{\bar a_1}=\frac{1}{2R}(a_1+\bar a_1),\ \ \ A\frac{e^{i\theta_0}}{a_1}=\frac{R^2-1}{R(1+R^2)^2}(a_1+\bar a_1).
\end{equation}
From (6), (7) and (10) we get that $(1+R^2)^2((a_1^2-\bar a_1^2)=0,$ as  $A\frac{e^{i\theta_0}}{a_1}\neq 0$ , then
$$a_1=\bar a_1,\ \  Ae^{i\theta_0}=2a_1^2\frac{R^2-1}{R(1+R^2)^2}>0,\ \ e^{i\theta_0}=1.$$

Applying it to equations (6,7,9), and with a simple computation we have that
$a_2=\bar a_2,$ and  $ \Psi=\bar \Psi.$  We take the derivatives of coordinates functions with respect to $\theta$, at point $P$,
$$ X_{\theta \theta \theta}=i\frac{A}{2}\frac{e^{i\theta_0}}{a_1}[(\bar \Psi-\Psi)(1-R^2)+2R(\bar a_2-a_2)+a_1^2-\bar a_1^2]=0,$$
$$Z_{\theta \theta \theta}=iA\frac{e^{i\theta_0}}{a_1}[(\bar \Psi-\Psi)R+ a_2-\bar a_2]=0.$$
Hence at point $P$,
$$\det \left(
    \begin{array}{ccc}
      X_\theta &  X_{\theta \theta} &  X_{\theta \theta \theta }\\
      Y_\theta &  Y_{\theta \theta} &  Y_{\theta \theta \theta }\\
     Z_\theta &  Z_{\theta \theta} &  Z_{\theta \theta \theta } \\
    \end{array}
  \right)=\det \left(
    \begin{array}{ccc}
      0 &  X_{\theta \theta} &  0\\
      Y_\theta &  Y_{\theta \theta} &  Y_{\theta \theta \theta }\\
     0 &  Z_{\theta \theta} &  0 \\
    \end{array}
  \right)=0.$$
Thus the torsion of boundary curve at point $P$ vanishes, i.e.
$$\tau(w_0)=\frac{\det(U_\theta,U_{\theta\theta},U_{\theta\theta\theta})}{|U_\theta \times U_{\theta\theta}|^2}=0.$$
Since $P$ is any point on the boundary curves, then the boundary curves are therefore torsion free, and which are planar circles.
\end{proof}

\section{The Proof of the Theorem}

In this section, by using the fourier series representation of the annulus we prove the theorem, i.e.
 \begin{thm}
 A free boundary minimal immersed annulus in unit ball ${\mathbb{B}}^3$ in euclidean 3-space is congruent with the critical catenoid.
\end{thm}
\begin{proof}
 Recall the Weierstrass representation of the minimal annulus. From section 4, and $e^{i\theta_0}=1$ in section 5, we have
 $$U:\Omega\rightarrow R^3, z\rightarrow U(z)=(X(z),Y(z),Z(z))$$
\begin{align*}
X(z)&=\text {Re}\int^z_{z_0}(1-g(z)^2)\frac{A}{2}\frac{1}{g_zz^2}dz+X_0,\\
Y(z)&=\text {Re}\int^z_{z_0}i(1+g(z)^2)\frac{A}{2}\frac{1}{g_zz^2}dz+Y_0,\\
Z(z)&=\text {Re}\int^z_{z_0}2g(z)\frac{A}{2}\frac{1}{g_zz^2}dz+Z_0.
\end{align*}
We know that $\frac{1}{g_z}, \ \frac{g^2}{g_z}$ and  $\frac{g}{g_z}$ are all holomorphic functions on the domain $\Omega$. Take the Laurent series of these functions as
$$
\frac{1}{g_z}=\sum^\infty_{k=-\infty}a_kz^k,\ \ \ \frac{g^2}{g_z}=\sum^\infty_{k=-\infty}b_kz^k,\ \ \ \frac{g}{g_z}=\sum^\infty_{k=-\infty}c_kz^k.
$$

Concretely we get
\begin{align}
 X(z)&=\frac{A}{2}\text {Re} \ \{\sum_{k\neq 0} \frac{a_{k+1}-b_{k+1}}{k}z^k+(a_1-b_1)\ln z\}+X_0, \\
Y(z)&=\frac{A}{2}\text {Re} \ i \{\sum_{k\neq 0} \frac{a_{k+1}+b_{k+1}}{k}z^k+(a_1+b_1)\ln z\}+Y_0, \\
Z(z)&=\frac{A}{2}\text {Re} \ \{\sum_{k\neq 0} \frac{2c_{k+1}}{k}z^k+2c_1\ln z\}+Z_0.
\end{align}
Here, $\bar a_1=- b_1$ and $c_1$ should be real number insuring the coordinate functions to be single valued.
The boundary curves of the minimal annulus are circles, denoted by $\partial \Omega(R)$ and $\partial \Omega(R^{-1})$. Because the coordinate functions of the  minimal annulus are harmonic functions, then by Green's formula and boundary orthogonal condition we have
$$\int_{\Omega}\Delta UdA=\int_{\partial \Omega }\frac{\partial U}{\partial \nu}ds=\int_{\partial \Omega }Uds=0,$$
here $\nu $ is the outward unit normal vector field on boundaries of the annulus, and hence we have
$$\int_{\partial \Omega(R) }Uds=-\int_{\partial \Omega(R^{-1})}Uds.$$
So the centers of the boundary circles are in opposite direction with respect to the center of the unit ball ${\mathbb{B}}^3$, which means the boundary circles are parallel to each other.

Without loss of generality, we can put the centers of the boundary circles on the Z axis by a rotation of the ball $\mathbb{B}^3$, then the coordinates $Z(R,\theta)$ and $Z(R^{-1},\theta)$ of the boundary circles are constant. From equation (13) we get the fourier series of $Z(r,\theta)(\ r=R, R^{-1})$  as bellow,
$$Z(r,\theta)=\frac{A}{2}\sum_{k\neq 0}\frac{1}{2k}(2c_{k+1}r^k-2\bar c_{-k+1}r^{-k})e^{ik\theta}+Ac_1\ln r+Z_0.$$
As $Z(r,\theta)(\ r=R, R^{-1})$ are constant, then the coordinates of the $e^{ik\theta}$'s should vanish, i.e. $$c_{k+1}r^k-\bar c_{-k+1}r^{-k}=0, \ r=R, R^{-1}; \ k=\pm 1,\pm 2, \cdots,
 $$ which are equivalent to $c_{k+1}=0,\  k=\pm 1,\pm 2, \cdots$. Thus
 $$\frac{g}{g_z}=\sum^\infty_{k=-\infty}c_kz^k=c_1z.$$
 From here we know $c_1\neq 0$, otherwise, $g=0$ and the annulus is flat, which is a contradiction. Hence
 $$g=cz^m,\ \ m=\frac{1}{c_1},\ \  c\neq 0,$$
  $c$ is a constant number, and $m\neq 0$. Note that $g$ is a simple valued function defined on the domain $\Omega$, then $m$ should be an integral number. Through a simple computation, we have
  \begin{align*}
\frac{1}{g_z}=\sum^\infty_{k=-\infty}a_kz^k=\frac{c_1}{c}z^{1-m},\ \ a_{1-m}=\frac{c_1}{c},\\
\frac{g^2}{g_z}=\sum^\infty_{k=-\infty}b_kz^k=cc_1z^{1+m},\ \ b_{1+m}=cc_1.
\end{align*}

 Since $m\neq 0$, therefore $a_1=0,\ b_1=0$. Then we have the coordinate functions of the boundary circles
\begin{align*}
X(r,\theta)&=X_0-\frac{Ac_1}{4m}(\phi e^{mi\theta}+ \bar \phi e^{-mi\theta }), \\
Y(r,\theta)&=Y_0+\frac{Ac_1}{4m}i(\phi e^{mi\theta}- \bar \phi e^{-mi\theta }), \\
Z(r,\theta)&=Z_0+Ac_1\ln r,
\end{align*}
where $\phi=cr^m+\bar c^{-1}r^{-m},\ \ r=R,\ R^{-1}$. If $|m|\geq 2$, the annulus is not embedded, which contradicts our assumption, then $m=\pm 1$, hence $c_1=\pm 1$. Compute the square of radius of the boundary circles, $X^2+Y^2$, which are constant and  independent of $\theta$, consequently we have $X_0=Y_0=0$. Then the annulus can be written as $$U: \Omega \rightarrow R^3, \ (r,\theta) \rightarrow U(r,\theta),$$
$$U(r,\theta)=-A(\cosh (\ln l^{\pm 1}r)\cos (\theta \pm \beta), \cosh (\ln l^{\pm 1}r)\sin (\theta \pm \beta), \mp\ln r-Z_0/A),$$
here the constant $c$ is taken in the form $c=le^{i\beta}$. Which implies that the minimal annulus is contained in a Catenoid, it must be the critical Catenoid [1]. The theorem is proved.
\end{proof}

\section*{Acknowledgment}
We would like to express gratitude to Professor Lei Ni, for his suggestions and discussions, and also thanks to Professor Hailiang Li, Zhenlei Zhang and Associate Professor Shicheng Xu for their useful discussions. This work was partially supported by NSFC (Natural Science Foundation of China) No. 11771070.

\end{document}